\documentclass[11pt,twoside]{amsart}
\usepackage[all]{xy}
        \usepackage{amssymb,latexsym,amsthm,amsmath,mathtools,hyperref,comment,epsfig,amscd}
        \usepackage{enumitem,color}
        
\hypersetup{
	colorlinks=true,
	linkcolor=blue,
	filecolor=magenta,      
	urlcolor=cyan,
	citecolor=red,
}

\usepackage{mathtools}
\DeclarePairedDelimiter\ceil{\lceil}{\rceil}

\long\def\symbolfootnote[#1]#2{\begingroup%
\def\thefootnote{\fnsymbol{footnote}}\footnote[#1]{#2}\endgroup}

\newcommand{\tr}{\ensuremath{{}^t\!}}

\newcommand{\F}{\mathbb{F}_q}
\newcommand{\Aut}{\textup{Aut}}

\newcommand{\bigzero}{\mbox{\normalfont\Large\bfseries 0}}

\makeatletter
\def\imod#1{\allowbreak\mkern10mu({\operator@font mod}\,\,#1)}
\makeatother

\newtheorem{theorem}{Theorem}[section]
\newtheorem{lemma}[theorem]{Lemma}
\newtheorem{corollary}[theorem]{Corollary}
\newtheorem{proposition}[theorem]{Proposition}
\newtheorem*{theorem*}{Theorem}
\theoremstyle{definition}
\newtheorem{definition}[theorem]{Definition}

\newtheorem{example}[theorem]{Example}
\numberwithin{equation}{section}
\newcommand{\ignore}[1]{}

\newcommand{\mynote}[1]{}

\begin{document}
\setcounter{section}{0}
\title{Nilpotent Lie Algebras of breadth type $(0,3)$}
\author{Rijubrata Kundu}
\address{Indian Institute of Science Education and Research Mohali, Knowledge City, Sector 81, Mohali 140 306, India}
\email{rijubrata8@gmail.com}
\author{Tushar Kanta Naik}
\address{Indian Institute of Science Education and Research Mohali, Knowledge City, Sector 81, Mohali 140 306, India}
\email{mathematics67@gmail.com}
\author{Anupam Singh}
\address{Indian Institute of Science Education and Research Pune, Dr. Homi Bhabha Road, Pashan, Pune 411 008, India}
\email{anupamk18@gmail.com}
\thanks{Rijubrata Kundu thanks IISER Mohali for the Institute Post-Doctoral Fellowship. Tushar 
Kanta Naik would like to acknowledge support from NBHM via grant 0204/3/2020/R\&D-II/2475. Anupam 
Singh is funded by SERB through CRG/2019/000271 for this research.}
\today

\begin{abstract}
For a natural number $m$, a Lie algebra $L$ over a field $k$ is said to be of breadth type $(0, 
m)$ if the co-dimension of the centralizer of every non-central element is of dimension $m$. In 
this article, we classify finite dimensional nilpotent Lie algebras of breadth type $(0, 3)$ over 
$\mathbb F_q$ of odd characteristics up to isomorphism. We also give a partial classification of 
the same over finite fields of even characteristic, $\mathbb C$ and $\mathbb R$. We also discuss 
$2$-step nilpotent Camina Lie algebras.
\end{abstract}

\maketitle

\section{Introduction}

The classification of finite dimensional nilpotent Lie algebras is an important topic in the subject 
of Lie theory. Over a finite field of characteristics $p$ it shares a lot of similarities with that 
of $p$-groups. Nilpotent Lie algebras of small dimensions are classified (see for 
example~\cite{dg, jnp, se}). To study the same in higher 
dimensions we use breadth and breadth type for Lie algebras which are analogous to the 
notion of breadth and conjugate type for finite $p$-groups (see for example~\cite{ch, is1, 
nky}) which have been useful there. Let $L$ be a finite dimensional Lie algebra over a field $k$. 
The breadth of an element $x\in L$, denoted by $b(x)$, is the co-dimension of $C_L(x)$, the 
centralizer of $x$ in $L$. The {\bf breadth of a Lie algebra} $L$, denoted by $b(L)$, is the maximum 
of the breadths of all its elements. More generally, we can define breadth type for a Lie algebra. A 
finite dimensional Lie algebra $L$ is said to be of {\bf breadth type} $(0=m_0, m_1, m_2,\cdots, 
m_r)$ where $m_i$ are the distinct breadth of elements of $L$ written in increasing order (that is, 
$m_i < m_{i+1}$). In this paper we ask the following questions: 
\begin{enumerate}
\item[(a)] Given positive integers $\{0=m_0, m_1, m_2,\cdots, m_r\}$, does there exist a nilpotent 
Lie algebra of breadth type $(0=m_0, m_1, m_2,\cdots, m_r)$? 
\item[(b)] Classify all finite dimensional nilpotent Lie algebras of breadth type $(0, m)$ for a 
positive integer $m$.    
\end{enumerate}
Barnea and Isaacs (see~\cite{bi}) studied Lie algebras with few centralizer dimensions. They proved 
that Lie algebras of breadth type $(0, m)$ need not be nilpotent, not even solvable. However, any 
nilpotent Lie algebra of breadth type $(0, m)$ has nilpotency class at most $3$. Among 
several interesting results, they classified non-nilpotent Lie algebras of type $(0, m)$ over 
$\mathbb C$.  In~\cite{kms, swk, re}, the nilpotent Lie algebras of breadth $1, 2$, 
and $3$ are studied. The nilpotent Lie algebras of breadth type $(0,1)$ (see 
Section~\ref{01type}) and $(0,2)$ (see Section~\ref{02type}) are easy to get from these results. We 
remark that the results in this subject do depend on the base field as we will see in this article 
too.  

In this paper, we classify nilpotent Lie algebras of breadth type $(0, 3)$ over certain fields. Let 
$\mathcal{L}_m$ be the free $2$-step nilpotent Lie algebra obtained as the quotient of the free Lie 
algebra $FL_{m+1}$ on $m+1$ generators by $[FL_{m+1},\; [FL_{m+1},FL_{m+1}]]$ (see 
Section~\ref{correspondence} for more details). The main theorem of this paper is as follows:
\begin{theorem}\label{breadth_type_(0,3)_classification_odd_char_finite_fields}
Let $L$ be a finite dimensional nilpotent Lie algebra over finite field $\mathbb F_q$ of odd 
characteristic. Then, $L$ is of breadth type $(0,3)$ if and only if $L \cong \mathfrak{g}+ I$, where 
$I$ is some Abelian Lie algebra, and $\mathfrak{g}$ is one of the following:
\begin{itemize}
\item[(i)] a $2$-step Camina Lie algebra with derived subalgebra of dimension $3$,
\item[(ii)] the Lie algebra $\mathcal{L}_3$ generated by four elements $x_1,x_2,x_3,x_4$.
\item[(iii)] the quotient $\mathcal{L}_3 / I$, where $I$ is the central ideal of dimension $1$ 
given by $I = \langle [x_1,x_2] + [x_3,x_4]\rangle$,
\item[(iv)] the quotient $\mathcal{L}_3 / J$, where $J$ is the central ideal of dimension $2$ 
given by $J = \langle [x_1,x_2] + [x_3,x_4],\; [x_1,x_3] + t[x_2, x_4]\rangle$, where $t$ is a 
non-square in $\F$.
\end{itemize}
\end{theorem}
\noindent We prove this theorem in Section~\ref{proof_of_the_main_theorem}. 

In Section~\ref{correspondence} we explore a correspondence between $p$-groups of nilpotency class 
$2$ of exponent $p$ and finite dimensional $2$-step nilpotent Lie algebras over $\mathbb F_p$ where 
$p$ is an odd prime. Using this we give positive answer to our first question over the field 
$\mathbb F_p$, when $p$ is odd. Further, using the known results from ~\cite{ny} about 
$p$-groups of conjugate type $(1,p^3)$ we can get our 
Theorem~\ref{breadth_type_(0,3)_classification_odd_char_finite_fields} over $\mathbb F_p$. The 
proof over a more general field requires more effort which we do in the later sections. In 
Section~\ref{camina_Lie_algebra} we include a study of $2$-step nilpotent Camina Lie algebras as it 
is required in our work. Although the classification of these Lie algebras is difficult to get, we 
manage to deduce results over some fields. With our analysis we also get classification of 
$4$-generated $2$-step nilpotent Lie algebras of breadth type $(0, 3)$ over finite fields of 
characteristic $2$ (see 
Theorem~\ref{four_generated_nilpotent_Lie_algebra_(0,3)_even_characteristic}), over $\mathbb C$ (see 
Theorem~\ref{four_generated_nilpotent_Lie_algebra_(0,3)_complex}), and over $\mathbb R$ (see 
Theorem~\ref{four_generated_nilpotent_Lie_algebra_(0,3)_real}). As mentioned earlier, this also 
highlights that the classification depends on the field.

\subsection{Notations}
We set some notations that will be used throughout the paper. A Lie algebra is assumed to be finite 
dimensional over a field $k$. For a Lie algebra $L$, the center of $L$ is denoted by $Z(L)$ and the 
derived subalgebra $[L, L]$ by $L'$. A $c$-step nilpotent Lie algebra is a nilpotent Lie algebra of 
nilpotency class $c$. Finally, when we say that a Lie algebra $L$ is generated by $n$ elements, we 
mean that $L$ is minimally generated by $n$ elements.

\subsection*{Acknowledgement} We thank Pradeep Kumar Rai for his interest in this problem. 

\section{Breadth type and Nilpotent Lie algebras}\label{Breadth_type_(0,1)_(0,2)_Lie_algebras}

In this section, we recall some basic definitions and collect some easy to get results which sets 
tone for the rest of the paper. Let $L$ be a finite dimensional Lie algebra over a field $k$. The 
breadth $b(x)$ of an element $x\in L$, which is co-dimension of the centralizer of $x$ in $L$, can 
also be thought of as the rank of the adjoint map $ad_x\colon L \rightarrow L$ given by $y \mapsto 
[x,y]$. The breadth $b(L)$ of a Lie algebra $L$ is the maximum of the breadths of all its elements. 
The concept of breadth in a Lie algebra is analogous to the notion of breadth in $p$-groups which 
has been useful in understanding the classification of these groups. 
More generally, for a Lie algebra $L$ the breadth type captures all possible co-dimensions of 
various distinct centralizers, i.e., all possible breadths of elements of $L$. This notion is 
similar to the conjugate type in groups (see~\cite{it}). Considerable amount of literature is 
available for groups, in particular for $p$-groups, of small breadth and conjugate types (see for 
instance, \cite{it1},  \cite{it2}, \cite{is}, \cite{is1}, \cite{ps}, \cite{wi}, \cite{ny}, 
\cite{nky}, \cite{dj}, \cite{do}, \cite{cc}). 

We begin with some basic examples.
\begin{example}
Consider the Lie algebra $L=\mathfrak{sl}_2(k)$ where characteristics of $k$ is $\neq 2$. Then, 
$b(L)=2$ and 
breadth type is $(0, 2)$.  
\end{example}
\begin{example}\label{classical-heisenberg}
Let $\mathcal H= \left\{ \begin{bmatrix} 0 &a &c\\ 0&0&b \\ 0&0&0 \end{bmatrix} \mid a, b, c\in k 
\right\}$ be the classical Heisenberg Lie algebra  over a field $k$. The breadth type of $\mathcal 
H$ is 
$(0,1)$.  
\end{example}
\begin{example}\label{non-nilpotentCamina}
Let $k$ be a field, and $L=\langle x, y \rangle$ with $[x,y]=x$. Then, $L$ is an example of Camina Lie algebra (see Section~\ref{camina_Lie_algebra} for definition) but not nilpotent. Further, $L$ is of breadth type $(0,1)$. 
\end{example}

Now we see that breadth and breadth type for Lie algebras remain same up to isoclinism. We recall 
this notion here for the sake of completeness. 

\subsection{Isoclinism of Lie algebras}\label{isoclinism}
In 1940, Hall (see~\cite{ha}) introduced the concept of isoclinism of groups which is weaker 
than isomorphism, and plays a very important role in the classification of finite $p$-groups. 
In 1994, Moneyhun in~\cite{mo} introduced the notion of isoclinism of Lie algebras in a similar 
way. Let $L$ be a Lie algebra over a field $k$. We denote the central quotient as $\overline{L} = 
L/Z(L)$. Then, the adjoint map induces a map $a_{L} \colon \overline{L} \times \overline{L} 
\rightarrow  L'$ given by 
$$\left(x+Z(L),y+Z(L)\right) \mapsto [x,y].$$ 
Two Lie algebras $L$ and $K$ are said to be {\bf isoclinic} if there exists an  isomorphism $\phi$ 
of the quotient Lie algebra $\overline L = L/Z(L)$ onto $\overline{K} = K/Z(K)$, and an isomorphism 
$\theta$ of the derived subalgebras $L'$ onto  $K'$ such that the following diagram commutes:
\[
\begin{CD}
\overline L \times \overline L    @> {\phi\times\phi} >> \overline K \times \overline K \\
@V{a_L}VV        @VV{a_K} V\\
 L'  @>{\theta} >> K'.
\end{CD}
\]

It is easy to check that isoclinism is an equivalence relation among Lie algebras. Equivalence 
classes under this relation are called {\bf isoclinism families}. A Lie algebra $L$ is said to be 
{\bf stem (or pure)} if $Z(L) \subseteq L'$. It turns out that, each isoclinism 
family of finite dimensional Lie algebra contains a stem Lie algebra which is of minimal 
dimension. Further, if $L$ and $K$ are two isoclinic Lie algebras of the same dimension, then $L$ and $K$ are isomorphic. This is slightly different from the group situation as in the case of groups, 
two non-isomorphic groups of the same order can be isoclinic. For example, the extra-special 
$p$-groups of exponent $p$ and $p^2$ (of the same order) are isoclinic but they are not isomorphic. 

For a Lie algebra $L$ it turns out that $L$ is isoclinic to $L\oplus A$ for any Abelian 
Lie algebra $A$. Further, if a finite dimensional Lie algebras $L$ of dimension $m$ is isoclinic to 
a stem Lie algebra $K$ of dimension $n$ then $L$ is isomorphic to $K\oplus A$ for an Abelian Lie 
algebra $A$ of dimension $m-n$. We have,
\begin{proposition}
Let $L$ and $K$ be two isoclinic finite dimensional Lie algebras. Then $L$ and $K$ are of the same 
breadth type.
\end{proposition}
\noindent In the light of above, for the classification of Lie algebras of a given breadth 
type, it is enough to consider stem Lie algebras of that breadth type.

\subsection{Classification of $(0,1)$ type}\label{01type}

As an example, we discuss the classification of breadth type $(0,1)$ nilpotent Lie algebras. 
Let us begin with an example which is a generalisation of $3$-dimensional Heisenberg Lie algebra in 
Example~\ref{classical-heisenberg}.  

\begin{example}[Heisenberg Lie algebra of dimension $2m+1$]\label{usual-Heisenberg-Lie-algebra}
For $m\geq 1$, consider the Lie algebra $\mathcal{H}_m = \left\{\begin{bmatrix} 0 & \textbf{a} & c\\
0 & 0_{m} & \textbf{b}\\ 0 & 0 & 0 \end{bmatrix}\right\}$ consisting of matrices over a field 
$k$ where $a$ is a $1\times m$ row vector, $b$ is a $m\times 1$ column vector, and $0_{m}$ is the 
zero matrix of dimension $m\times m$. Clearly, $\mathcal H_1=\mathcal H$ and the center 
$Z(\mathcal{H}_m)$ is the derived subalgebra $\mathcal{H}_m'$ of dimension $1$. It is 
easy to check that $\mathcal{H}_m$ is a Camina Lie algebra (more on this in 
Section~\ref{camina_Lie_algebra}) of nilpotency class $2$ and breadth $1$. In general, we will 
call this $2m+1$-dimensional Heisenberg Lie algebra simply as {\bf Heisenberg Lie algebras}. A 
presentation of Heisenberg Lie algebra is as follows:  
$\mathcal{H}_m=\text{span}\{x_1, y_1, x_2, y_2, \ldots, x_m, y_m, z\}$ where 
$[x_i, y_j] = \delta_{ij} z$, $[z,L]=0$ (here $\delta_{ij}$ denotes the Kronecker delta). This 
presentation could be realized as follows:
$$x_i= \begin{bmatrix} 0 & e_i^t & 0\\ 0 & 0_{m} & 0\\ 0 & 0 & 0 \end{bmatrix}, \;\; 
y_j=\begin{bmatrix} 0 & 0 & 0\\ 0 & 0_{m} & e_j\\ 0 & 0 & 0
\end{bmatrix} , \; \; 
z=\begin{bmatrix} 0 & 0 & 1\\ 0 & 0 & 0\\ 0 & 0 & 0 \end{bmatrix}
$$
where $e_i$ is the usual $m\times 1$ column vector with $1$ in $i^{th}$ place and $0$ elsewhere. 
The breadth type of $\mathcal{H}_m$ is $(0,1)$
\end{example}

Nilpotent lie algebras of breadth type $(0,1)$ and breadth $1$ are same, thus we simply quote the 
classification of latter from~\cite{kms}, Theorem 2.4.
\begin{theorem}\label{breadth_type_(0,1)_classification}
Let $L$ be a finite dimensional nilpotent stem Lie algebra of breadth type $(0,1)$. Then, $L$ is 
isomorphic to a Heisenberg Lie algebra $\mathcal{H}_m$.
\end{theorem}
\subsection{Non-nilpotent Lie algebras over $\mathbb C$}
We finish this section by noting down~\cite{bi} Theorem A, which classifies non-nilpotent Lie 
algebras over $\mathbb C$ of type $(0, m)$.
\begin{theorem}
Let $L$ be a non-nilpotent complex Lie algebra of breadth type $(0,m)$. Then, $\dim(L/Z(L))\leq 
3$ and one of the following occurs:
\begin{enumerate}[label=(\alph*)]
\item $L/Z(L)$ is isomorphic to the unique $2$-dimensional non-abelian Lie algebra.
\item $L/Z(L)$ is isomorphic to $\mathfrak{sl}(2,\mathbb{C})$.
\item $L/Z(L)$ is isomorphic to the Lie algebra with basis $\{a,x,y\}$ and Lie brackets given by 
$[a,x]=x, [a,y]=-y, [x,y]=0$.
\end{enumerate}
\end{theorem}
\noindent Thus, we see that $m$ is necessarily $1$ or $2$ in this case. In case $(a)$, $L$ is of 
breadth type 
$(0,1)$. In case $(b)$ and $(c)$, $L$ is necessarily of breadth type $(0,2)$. On other hand, we 
will see that there are nilpotent Lie algebras of type $(0,m)$ for all $m\geq1$ (see for example 
$\mathcal L_m$ in Proposition~\ref{Lmtype0m} and Camina Lie algebras in 
Section~\ref{camina_Lie_algebra}).

\section{$p$-groups of exponent $p$ and nilpotent Lie algebras over 
$\mathbb{F}_p$}\label{correspondence}

In this section we explore a correspondence between $p$-groups of nilpotency class $2$, exponent 
$p$ and finite dimensional $2$-step nilpotent Lie algebras over the finite field $\mathbb{F}_p$ 
when $p$ is an odd prime. This seems to be known to experts but we could not find any 
reference for the same. We will see that this correspondence behaves well with respect to the 
conjugate type of a group and the breadth type of a Lie algebra. 

For a group $G$ one can associate a Lie ring $L(G)$ as follows (See Chapter 6,~\cite{kh}): 
$$\displaystyle L(G)=\bigoplus_{i} \gamma_i(G)/\gamma_{i+1}(G)$$
with the Lie bracket given by $$[a+\gamma_{i+1}(G), b+\gamma_{j+1}(G)]=[a,b]+\gamma_{i+j+1}(G)$$
where $[a,b]$ is the commutator of $a$ and $b$ in $G$. When $G$ is a $p$-group, $L(G)$ is a vector 
space over $\mathbb{F}_p$, and if $G$ is of nilpotency class $2$ then $L(G)=(G/G')\bigoplus G'$. 
The Lie bracket, in this case, is simply $[(a+G')+x,(b+G')+y]= [a,b]$ where $a+G',b+G'\in G/G'$ and 
$x,y\in G'$. To set up the desired correspondence we make use of a set of minimal generators of the 
group. 

We begin with the following special $p$-group,
\begin{equation}\label{group_Gm}
\mathcal{G}_m = \left\langle g_1, g_2, \ldots, g_{m+1} \mid g_i^p = [[g_i, g_j],g_k] = 1,\; 1\leq i, 
j, k \leq m \right\rangle.
\end{equation}
Analogous to these groups, we consider the free nilpotent Lie algebras $\mathcal{L}_m$ over a field $k$ as follows:
\begin{equation}\label{Lie_algebra_Lm}
\mathcal{L}_m = \big{\langle} x_1, \ldots, x_{m+1} \mid [[x_i, x_j], x_k] = 0,\; 1\leq i, j, k \leq 
m \big{\rangle}.
\end{equation}
That is, $\mathcal{G}_m \cong F_{m+1} / \langle F_{m+1}^p, \gamma_3(F_{m+1}) \rangle$, and 
$\mathcal{L}_m \cong FL_{m+1} /  \gamma_3(FL_{m+1})$ where $F_{m+1}$ is the free group of rank 
$m+1$ and $FL_{m+1}$ is the $m+1$-generated free Lie algebra over a field $k$. Note that 
$\mathcal{G}_{m}$ and $\mathcal{L}_{m}$ both are nilpotent of nilpotency class $2$. Further, the 
cardinality of any minimal generating set of both $\mathcal{G}_m$ and $\mathcal{L}_m$ is $m+1$.

We note some properties of the Lie algebra $\mathcal{L}_m$.
\begin{proposition}\label{Lmtype0m}
Let $\mathcal L= \mathcal{L}_m$ be the free $2$-step nilpotent Lie algebra over a field $k$ generated by $(m+1)$-elements. Then, 	
\begin{enumerate}
\item the dimension of $\mathcal{L}$ is $\frac{(m+1)(m+2)}{2}$, and the center $Z(\mathcal{L})$ is equal to the derived subalgebra $\mathcal{L}'$ of dimension $\frac{m(m+1)}{2}$.
\item For each $x\in \mathcal{L} \setminus Z(\mathcal{L})$, the centralizer $C_{\mathcal{L}}(x) = \langle x,\; Z(\mathcal{L})\rangle$, and is of dimension  $1 + \frac{m(m+1)}{2}$.
\item $\mathcal{L}$ is of breadth type $(0,m)$.
\end{enumerate}
\end{proposition}
\noindent The following well-known properties of generating sets will be useful.
\begin{proposition}\label{automorphism_of_free_nilpotent_Lie_algebra}
Let $A$ and $B$ be two minimal generating set of $\mathcal{L}_m$. Then, any bijective map between 
$A$ and $B$ extends to an automorphism of $\mathcal{L}_m$.	
\end{proposition}
\noindent Consider a set of $(m+1)$ elements in $\mathcal{L}_m$, say  $b_1, b_2, \ldots, b_{m+1}$, 
expressed in terms of the generators,
$$b_j  =  \left( \sum_{i=1}^{m+1} {\alpha_{ij}} x_i \right)\; + h_j \hspace{2 cm} \text{ where }1\leq j\leq m+1$$
where $h_1,h_2,\ldots, h_{m+1}\in \mathcal{L}_m'$. Then,
\begin{proposition}\label{set_of_generators}
The set $\{b_1,b_2,\ldots,b_{m+1}\}$ generates $\mathcal{L}_m$ if and only if the matrix 
$(\alpha_{ij})$ is invertible.
\end{proposition}
\noindent All of the above propositions are easy to prove and hold for $\mathcal{G}_m$ as well with 
necessary changes. Now we consider $\mathcal L_m$ over $\mathbb F_p$. We point 
out that $\mathcal{G}_m/\mathcal{G}_m'$ and $\mathcal{L}_m/ \mathcal{L}_m'$, both can be seen as 
vector spaces over $\mathbb{F}_p$ of dimension 
$m+1$. Similarly, $Z(\mathcal{G}_m) = \mathcal{G}_m'$ and $Z(\mathcal{L}_m) = \mathcal{L}_m'$, both 
are vector spaces over $\mathbb{F}_p$ of dimension $\frac{m(m+1)}{2}$.

Now, we fix $m\geq 1$ and define a map $\psi \colon \mathcal{G}_m\rightarrow \mathcal{L}_m$ as 
follows. An arbitrary element $g\in \mathcal{G}_m$ can be written as $$\displaystyle g = \left( 
\prod_{i=1}^{m+1} g_i^{\alpha_i} \right)\; \left( \prod_{1\leq j < r \leq m+1} [g_j, 
g_r]^{\beta_{j,r}} \right),\;\; \text{for some}\; 0\leq \alpha_i,\; \beta_{j, r}\leq p-1.$$
Similarly, an arbitrary element $x \in \mathcal{L}_m$ can be written as $$\displaystyle x = \left( 
\sum_{i=1}^{m+1} \gamma_i x_i \right) +  \left( \sum_{1\leq j < r \leq m+1} \xi_{j, r} [x_j, x_r] 
\right),\;\; \text{for some}\;\gamma_i,\; \xi_{j, r}\in \mathbb{F}_p.$$
The map $\psi$ is given by sending $g\in \mathcal G_m$ as follows:
$$ \left( \prod_{i=1}^{m+1} g_i^{\alpha_i} \right) \left( \prod_{1\leq j < k\leq m+1} [g_j, g_k]^{\beta_{j,k}} \right) \mapsto \left( \sum_{i=1}^{m+1} \alpha_i x_i \right) +  \left(\sum_{1\leq j < k\leq m+1} \beta_{j,k}[x_j,x_k] \right).$$
\noindent This map $\psi$ has some interesting properties. It induces a vector space 
isomorphism $\psi_I \colon \mathcal{G}_m/\mathcal{G}_m' \rightarrow \mathcal{L}_m/ \mathcal{L}_m'$, 
and a vector space isomorphism $\psi_R \colon Z(\mathcal{G}_m) \rightarrow 
Z(\mathcal{L}_m)$ by restricting the map $\psi$. Moreover, the following diagram commutes:
\[
\begin{CD}
\mathcal{G}_m/\mathcal{G}_m' \times \mathcal{G}_m/\mathcal{G}_m'  @>\psi_I \times \psi_I >> 
\mathcal{L}_m/ \mathcal{L}_m' \times \mathcal{L}_m/ \mathcal{L}_m'\\
@V{\psi_1}VV        @VV{\psi_2}V\\
Z(\mathcal{G}_m) @>\psi_R>> Z(\mathcal{G}_m),
\end{CD}
\]
where $\psi_1 \colon \mathcal{G}_m/\mathcal{G}_m' \times \mathcal{G}_m/\mathcal{G}_m' \mapsto 
Z(\mathcal{G}_m)$, and $\psi_2 \colon \mathcal{L}_m/\mathcal{L}_m' \times 
\mathcal{L}_m/ \mathcal{L}_m' \mapsto Z(\mathcal{L}_m)$ are the commutator and the Lie bracket 
map, respectively.

The set bijection $\psi\colon \mathcal{G}_m \rightarrow \mathcal{L}_m$ induces a set bijection 
$\Psi \colon \Aut(\mathcal{G}_m) \rightarrow \Aut(\mathcal{L}_m)$ as follows. In the view of 
Proposition~\ref{automorphism_of_free_nilpotent_Lie_algebra} and~\ref{set_of_generators} we need to 
define this at the level of the generating sets. Suppose $\theta_{\mathcal{G}}$ is in 
$\Aut(\mathcal G_m)$ defined by $$\theta_{\mathcal{G}}(g_j)=\left( \prod_{i=1}^{m+1} 
g_i^{\alpha_{ij}} \right)\;\; \left( \prod_{1\leq r < s\leq m+1} [g_r,g_s]^{\beta_{rsj}} \right), 
\text{ for all } 1\leq j\leq m+1$$
then we define $\Psi(\theta_{\mathcal G})$ to be $\theta_{\mathcal{L}}$ as follows
$$\theta_{\mathcal{L}}(x_j)= \left(\sum_{i=1}^{m+1} \alpha_{ij}x_i \right) +  \left(\sum_{1\leq r < s\leq m+1} \beta_{rsj}[x_r,x_s] \right), \text{ for all } 1\leq j\leq m+1.$$
Let $N_1, N_2 \leq Z(\mathcal{G}_m)$ and $\theta \in \Aut(\mathcal{G}_m)$ such that $\theta(N_1) = 
N_2$. Then it follows that,
$$\Psi(\theta)(\psi_R(N_1)) = \psi_R(N_2).$$
\noindent The proof of the following two lemmas can de done easily by applying 
Proposition~\ref{automorphism_of_free_nilpotent_Lie_algebra} and Proposition~\ref{set_of_generators} 
and its group analogues.
\begin{lemma}\label{Isomorphism_between_central_quotients_extends to_an_automorphism-I}
Let $N_1, N_2$ be two central subgroups of $\mathcal G= \mathcal{G}_m$. Then, $\mathcal{G}/N_1 
\cong \mathcal{G}/N_2$ if and only if there exists an automorphism $\theta$ of $\mathcal{G}$ such 
that $\theta (N_1)=N_2$.
\end{lemma}
\begin{lemma}\label{Isomorphism_between_central_quotients_extends to_an_automorphism-II}
Let $I_1, I_2$ be two central ideals of $\mathcal{L}=\mathcal{L}_m$. Then, $\mathcal{L}/I_1 \cong 
\mathcal{L}/I_2$ if and only if there exists an automorphism $\theta$ of $\mathcal{L}$ such that 
$\theta (I_1) = I_2$.
\end{lemma}
\noindent Using these we conclude,
\begin{theorem}\label{Relation_between_central_quotients}
Let $N_1, N_2$ be two central subgroups of $\mathcal G=\mathcal{G}_m$. We set $I_1 = \psi_R (N_1)$, 
and $I_2 = \psi_R (N_2)$. Then, $\mathcal{G}/N_1 \cong \mathcal{G}/N_2$ if and only if 
$\mathcal{L}/I_1 \cong \mathcal{L}/I_2.$
\end{theorem}

We are now ready to give a bijection $\Phi$ between the families $\mathcal{PG}_m$ consisting of 
certain finite groups and $\mathcal{NL}_m$ consisting of certain finite dimensional Lie algebras, 
which are defined as follows:
\begin{eqnarray*}
\mathcal{PG}_m &=&\{ G \mid G \text{ is a $(m+1)$-generated } p\text{-group of nilpotency 
class } 2 \text{ and exponent } p\}.\\ 
\mathcal{NL}_m &=&\{ L \mid L \text{ is a $(m+1)$-generated 2-step nilpotent Lie algebra over } 
\mathbb{F}_p\}.
\end{eqnarray*}
We note that $\mathcal G_m\in \mathcal{PG}_m$ and $\mathcal L_m\in \mathcal{NL}_m$. The map $\Phi$ 
respects the map $\psi$. Let $G\in \mathcal{PG}_m$. Then, there exists a central subgroup $N \lneq 
Z(\mathcal{G}_m)$ such that $G \cong \mathcal{G}_m/N$. We map $G$ to $\mathcal L_m / \psi_R(N)$. We 
know that a Lie algebra $L\in \mathcal{NL}_m$ is of the form $L \cong \mathcal{L}_m/I$ for some 
central ideal $I \lneq Z(\mathcal{L}_m)$. Thus we have,
\begin{theorem}
There exists a one-one correspondence $\Phi$ between the families $\mathcal{PG}_m$ and $\mathcal{NL}_m$ given by 
$$\Phi(G) = \mathcal{L}_m/ \psi_R (N)$$ 
where $G \cong \mathcal{G}_m/N$ for some $N\lneq Z(\mathcal{G}_m)$. Moreover, if $G_1$ and $G_2$ 
are two groups in $\mathcal{PG}_m$ then $G_1 \cong G_2$ if and only if $\Phi (G_1) \cong \Phi 
(G_2)$.
\end{theorem}
\noindent This follows from Theorem~\ref{Relation_between_central_quotients}.

Further, the notion of conjugate type carries over to the notion of breadth type under this correspondence.
\begin{proposition}
A group $G\in \mathcal{PG}_m$ is of conjugate type $(1=p^0, p^{m_1}, p^{m_2}, \ldots, p^{m_r})$ if 
and only if the Lie algebra $\Phi (G)\in \mathcal{NL}_m$ is of breadth type $(0, m_1, m_2, \ldots, 
m_r)$.
\end{proposition}
\begin{proof}
It is enough to prove that if the size of a centralizer in $G$ is $p^l$, then there exists a 
centralizer in $\phi(G)$ whose dimension is $l$ and vice-versa. Since $G$ is of the form 
$\mathcal{G}_m /N$ for some central subgroup $N$, it is easy to see that the natural bijection from 
$\mathcal{G}_m$ to $\mathcal{L}_m$ maps a centralizer of an element of $\mathcal G_m/N$, say of 
size $p^l$, to the centralizer of its image in $\mathcal{L}_m$ which is also of size $p^l$. Hence 
it is of dimension $l$ over $\mathbb{F}_p$. The converse also follow from a similar argument.
\end{proof}

This correspondence is useful in classifying the $2$-step nilpotent Lie algebras of a certain 
breadth type if the classification of the same in terms of conjugate type is known for the $2$-step 
$p$-groups. For example, classification of breadth type $(0, 3)$ nilpotent Lie algebras 
over $\mathbb{F}_p$ can be obtained from the classification of $p$-groups of conjugate type 
$(1, p^3)$ (see~\cite{ny} for the classification of $(1,p^3)$ type). 
Over more general field $\mathbb F_q$, 
Theorem~\ref{breadth_type_(0,3)_classification_odd_char_finite_fields} still needs a proof which we 
do in Section~\ref{proof_of_the_main_theorem}.

Now, consider a set $\mathcal S = \{m_0=0, m_1, \ldots, m_r \}$ consisting of positive integers. 
In~\cite{ch}, it is proved that there is a $p$-group of nilpotency class $2$ and exponent $p$ of 
type $\{1=p^0, p^{m_1}, \ldots, p^{m_r}\}$ when $p$ is odd. We are ready to answer the following 
question: Does there exist a finite dimensional Lie algebra of breadth type $(m_0, m_1, \ldots, 
m_r)$? Our correspondence gives the following: 
\begin{corollary}\label{existence}
Given a set $\mathcal S = \{0=m_0, m_1, \ldots, m_r \}$ of positive integers there exists a 
$2$-step nilpotent Lie algebra over $\mathbb{F}_p$, $p$-odd, of breadth type $(m_0, m_1, 
\ldots, m_r)$. 
\end{corollary}
\noindent In general, the family of Lie algebras of breadth $m$ can be divided into $2^{m-1}$ sub 
families with respect to the breadth type. What we have shown here is that each subfamily is non 
empty over the finite field $F_p$, when $p$ is odd. One can ask the above question over other fields 
too.


\section{$2$-step nilpotent Camina Lie algebras}\label{camina_Lie_algebra}

In the previous section, we have seen that the free $2$-step nilpotent Lie algebra $\mathcal L_m$ 
(see Proposition~\ref{Lmtype0m}) over a field $k$ is of breadth type $(0, m)$. In this section, we 
introduce another class of $2$-step nilpotent Lie algebras of breadth type $(0, m)$, namely the 
Camina Lie algebras. Interestingly, we will see that the existence of such Lie algebra heavily depends 
on the field, unlike the other examples. This highlights one of the complications in this subject. 

A finite dimensional Lie algebra $L$ is said to be a {\bf Camina Lie algebra} if $[x, L] = L'$ for 
all $x \in L \setminus L'$ (see~\cite{ss} for more details). Not all Camina Lie algebras 
are nilpotent (see Example~\ref{non-nilpotentCamina}). Note that, if $L$ is a $2$-step 
nilpotent Camina Lie algebra then $Z(L)=L'$.  Furthermore, if $L'$ is of dimension $m$, then 
$L$ is of breadth type $(0, m)$. We have yet another generalisation of Heisenberg Lie algebra 
depending on the availability of a field extension of the given field $k$ which gives an example of 
Camina Lie algebra. 

\begin{example}[Heisenberg Lie algebra of degree $m$]\label{Heisenberg_algebra_degree_m}
Let $k$ be a field and $K$ be a field extension of degree $m$. Consider the Lie algebra $\mathfrak{h}_m = \left\{\begin{bmatrix} 0&a&c\\ 0&0&b\\ 0&0&0 \end{bmatrix} \mid a,b,c\in K \right\}$ over $k$. The Lie algebra $\mathfrak h_m$ is of dimension $3m$ with $Z(\mathfrak{h}_m) = \mathfrak{h}'_m$ of dimension $m$. Further, $[x,\mathfrak{h}_m]=\mathfrak{h}_m'$ for all $x\in \mathfrak{h}_m\setminus \mathfrak{h}'_m$. Thus, $\mathfrak{h}_m$ is a 2-step nilpotent Camina Lie algebra over $k$ of breadth type $(0,m)$. Note that when $m=1$, $\mathfrak h_1 = \mathcal H$. We call $\mathfrak h_m$ the Heisenberg Lie algebra of degree $m$ over $k$.
\end{example}

\begin{example}\label{unbounded_generators_Camina}
Consider the quotient Lie algebra $\mathfrak{h}_m/I$ where $I \lneq Z(\mathfrak{h}_m)$ is a central 
ideal of dimension $l < m$. Then, $\mathfrak{h}_m/I$ is a Camina Lie algebra of breadth type $(0, 
m-l)$. Further, $\mathfrak{h}_m/I$ is minimally generated by $2m$ elements just like 
$\mathfrak{h}_m$. 
\end{example}

Let $k$ be a field and $M_n(k)$ be the set of all $n\times n$ matrices over $k$. 
\begin{definition}
A vector subspace $V$ of $M _n(k)$ is said to be a {\bf rank-$n$ subspace} if every element of $V^{\times} = V\setminus \{0\}$ has rank $n$ (as matrices). That is, every element of $V^{\times}$ is invertible.
\end{definition}
\noindent Now we will explore the $2$-step nilpotent Camina Lie algebras  of breadth type $(0, m)$ 
in more detail and try to classify them over certain fields, for example, algebraically closed 
field, $\mathbb R$, and $\mathbb F_q$. For this, we will make use of the following theorem which is 
essentially a restatement of~\cite{ss}, Theorem 3.1. 
\begin{theorem}\label{structure_constant_Camina_algebra}
Let $L$ be a Lie algebra over $k$ minimally generated by $n$ elements. Suppose $Z(L)=L'$ of dimension $m$. Then, the following are equivalent:
\begin{enumerate}
\item $L$ is a Camina Lie algebra.
\item There exists a rank $n$-subspace $V$ of $M_n(k)$ consisting of skew-symmetric matrices with $\dim(V)=m$.
\end{enumerate}
\end{theorem}
\begin{proof} 
Let $L$ be a Lie algebra over $k$. Suppose $Z(L)=L'$. Let $\{x_1, x_2, \ldots, x_n\}$ be a minimal generating set of $L$, and $\{y_1, y_2, \ldots, y_m\}$ be a basis of $L'$. Suppose that $X_r = (\gamma_{jir})$, the $n\times n$ matrices for $1\leq r\leq m$, obtained by writing $[x_i, x_j] = \gamma_{ij1}y_1+ \gamma_{ij2} y_2 + \cdots + \gamma_{ijn}y_n$, for $1\leq i, j \leq n$, where $\gamma_{ijr} \in k$. Then, from~\cite{ss} Theorem 3.1, the following are equivalent:	
\begin{enumerate}[label=(\alph*)]
\item $L$ is a Camina Lie algebra.
\item If $\xi_1,\xi_2, \ldots, \xi_m \in k$ are such that $\xi_1 X_1+ \xi_2X_2 + \cdots + \xi_mX_m$ is singular, then $\xi_1 = \xi_2 = \cdots = \xi_m=0$.
\end{enumerate}
Note that the matrices $X_1, X_2, \ldots, X_m$ are skew-symmetric, and the restatement of the second 
part is our statement.
\end{proof}
\noindent In the view of the above theorem, the problem of the existence of $n$-generated $2$-step 
nilpotent Camina Lie algebra of breadth type $(0,m)$ over $k$ is equivalent to producing a 
collection of (structure constants) matrices $X_1, \ldots, X_m$ over $k$ which are (a) 
skew-symmetric of size $n\times n$, and (b) whose any non-trivial linear combination is 
non-singular. Interestingly, this linear algebra problem is important in its own right, for example 
determining Radon-Hurwitz numbers (see~\cite{ra, hu}), and has received quite a bit of attention 
with several applications to other areas of mathematics. 

\subsection{Rank $n$-subspaces of skew-symmetric matrices and Camina Lie algebras}
For $n\in \mathbb{N}$ and a field $k$, let $k(n)$ denote the maximum possible dimension of a rank-$n$ subspace of $M_n(k)$. In other words, $k(n)$ is the maximum number of invertible $n\times n$ matrices $\{A_1, A_2, \ldots, A_{k(n)}\}$ over $k$ such that any non-trivial linear combination of these matrices is non-singular. The set of skew-symmetric matrices of $M_n(k)$, denoted as $SKS_n(k)$, form a vector space of dimension $\frac{n(n-1)}{2}$. Let $k_{sks}(n)$ denote the maximum dimension of a rank-$n$ subspace of $SKS_n(k)$. Let us first determine $k(n)$ and $k_{sks}(n)$ for an algebraically closed field.

\begin{proposition}\label{Radon-Hurwitz-algebraically-closed-field}
Let $k$ be an algebraically closed field. Then, $k(n)=1$. Moreover, 
\[k_{sks}(n) = \begin{cases} 0 & \text{when $n$ is odd}\\ 1 & \text{when $n$ is even.} \end{cases}
\]
\end{proposition}
\begin{proof}
Consider a rank-$n$ subspace $V$ of $M_n(k)$ such that $\dim(V)=m\geq 2$. Suppose $V$ is 
spanned by $X_1,X_2,\ldots,X_m$. Let us take the first two matrices $X_1,X_2$. We claim that there 
exists $\alpha, \beta$ both non-zero such that $\alpha X_1+\beta X_2$ is singular. If not, then 
$X_1^{-1}(\alpha X_1+\beta X_2)$ is non-singular for all $\alpha,\beta \neq 0$. This implies 
$\alpha\beta^{-1}+X_1^{-1}X_2$ is non-singular for all $\alpha,\beta \neq 0$. However, by using a 
non-zero eigen value of $X_1^{-1}X_2$, which exists as $k$ is algebraically closed, we get a 
contradiction. Thus, $\dim(V)=1$. Further, we note that a skew-symmetric matrix is singular 
when $n$ is odd. This completes the proof.
\end{proof}
\noindent As a consequence we have the following. 
\begin{theorem}\label{complex_Camina_Lie_algebra}
Let $L$ be a $2$-step nilpotent Camina Lie algebra over an algebraically closed field $k$. Then, $\dim(L')=1$. In particular, $L$ has breadth type $(0,1)$. 
\end{theorem}

We have the following result that we will need later.
\begin{proposition}\label{skew-symmetric construction}
Let $V$ be a rank-$n$ subspace of $M_n(k)$ for a field $k$. Suppose $\dim(V)=r$. Then, there exists rank-$2n$ subspace $W$ of $SKS_{2n}(k)$ with $\dim(W)=r$.
\end{proposition}
\begin{proof} Suppose $V=\text{span}\{X_1, X_2, \ldots, X_r\}$. For $1 \leq i \leq r$, we define $Y_i= \begin{bmatrix} 0 & -X_i\tr \\ X_i & 0\end{bmatrix}$ in $M_{2n}(k)$.  Clearly, $Y_1,Y_2,\ldots Y_r$ span a $r$-dimensional rank-$2n$ subspace of $SKS_{2n}(k)$. This completes the proof.
\end{proof}

Now we determine $\mathbb{R}(n)$ and $\mathbb{R}_{sks}(n)$. The number $\mathbb{R}(n)$ is called Radon-Hurwitz number as they were classical treated by Radon in~\cite{ra} and Hurwitz in~\cite{hu} independently. We note down the formula for $\mathbb{R}(n)$ given in~\cite{alp}.
\begin{theorem}\label{Radon-Hurwitz_real_fields}
Let $n\in \mathbb{N}$. Write $n=2^{4a+b}.c$ where $c$ is odd and $0\leq b\leq 3$. Then, 
$\mathbb{R}(n)=8a+2^b.$
\end{theorem}
\noindent Note that, when $n$ is odd $\mathbb{R}(n)=1$. This can be also deduced by the same 
elementary arguments made for algebraically closed fields as in 
Proposition~\ref{Radon-Hurwitz-algebraically-closed-field}. The table for $\mathbb{R}(n)$ for small 
values of $n$ (even) is as follows:
\begin{table}[ht]
	\centering
	\begin{tabular}{|c|c|c|c|c|c|c|c|c|c|c|}
		\hline
		$n$ &  2 & 4 & 6 & 8 & 10 & 12 & 14 & 16 & 18 & 20\\
		\hline
		$\mathbb{R}(n)$ & 2 & 4 & 2 & 8 & 2 & 4 & 2 & 9 & 2 & 4\\
		\hline
	\end{tabular}
	\vspace{0.2 cm}
	\caption{Radon-Hurwitz sequence for some small values of $n$}
	\label{table: Radon-Hurwitz numbers}
\end{table}

\begin{example} Consider the matrices $ A= \begin{bmatrix} 1 & 0\\ 0 & 1 \end{bmatrix}$ and $B=\begin{bmatrix} 0 & -1\\ 1 & 0 \end{bmatrix}$ in $M_2(\mathbb R)$. Then, 
$ \alpha A+\beta B = \begin{bmatrix} \alpha & -\beta\\ \beta & \alpha \end{bmatrix}$. Then, $\det(\alpha A +\beta B)=\alpha^2+\beta^2$. Thus, $\alpha A +\beta B$ is non-singular when $(\alpha, \beta) \neq (0,0)$. This exhibits a $2$-dimensional rank-$2$ subspace of $M_2(\mathbb{R})$ as we have $\mathbb R(2)=2$. Further, since skew-symmetric matrices are of dimension $1$, we get $\mathbb{R}_{sks}(2)=1$.
\end{example}

From these results and Proposition~\ref{skew-symmetric construction}, it is now easy to construct 
$2$-step nilpotent real Camina Lie algebra of breadth type $(0,m)$. We demonstrate this using an 
example.
\begin{example}\label{Camina_real_eight_generated}
Consider a rank-$4$ subspace of $M_4(\mathbb{R})$ of dimension $3$. This is possible because of Theorem~\ref{Radon-Hurwitz_real_fields}. Now, by Proposition~\ref{skew-symmetric construction}, we can construct a rank-$8$ subspace of $SKS_8(\mathbb{R})$ of dimension $3$. Thus, in the view of Theorem~\ref{structure_constant_Camina_algebra}, we can produce a real $2$-step nilpotent Camina Lie algebra generated by $8$ elements and of breadth type $(0,3)$. 
\end{example} 

The following result from~\cite{ay} Theorem 3, connects the numbers $\mathbb{R}_{sks}(n)$ 
and $\mathbb{R}(n)$, in general.
\begin{theorem}[]\label{Radon-Hurwitz_skew_symmetric}
We have, $\mathbb{R}_{sks}(n)=\mathbb{R}(n)-1$.
\end{theorem}
\noindent Let us construct a $2$-step nilpotent real Camina Lie algebra with $4$ generators and breadth type $(0,3)$.
\begin{example}\label{real_Camina_four_generated}
By Theorem~\ref{Radon-Hurwitz_skew_symmetric}, we can construct a $3$-dimensional rank-$4$ subspace of $SKS_4(\mathbb{R})$. Take the following matrices
$$ X_1= \left(\begin{array}{@{}c|c@{}} \bigzero & 	\begin{matrix} 1 & 0 \\ 0 & 1 \end{matrix} \\ \hline \begin{matrix} -1 & 0 \\ 0 & -1 \end{matrix} & \bigzero \end{array}\right), \; X_2=\left(\begin{array}{@{}c|c@{}} \bigzero & 	\begin{matrix} 0 & -1 \\ 1 & 0 \end{matrix} \\ \hline \begin{matrix} 0 & -1 \\ 1 & 0 \end{matrix} & \bigzero \end{array}\right), \; 
X_3= \left(\begin{array}{@{}c|c@{}} \begin{matrix} 0 & -1 \\ 1 & 0 \end{matrix}  & 	\bigzero \\ \hline \bigzero &  \begin{matrix} 0 & 1 \\ -1 & 0 \end{matrix}   \end{array}\right).
$$
We can verify that $X_1^2 = X_2^2 = X_3^2= -I$. Using this we get, $(\alpha X_1 + \beta X_2 +\gamma X_3)(\alpha X_1 + \beta X_2 +\gamma X_3)\tr = -(\alpha^2+\beta^2+\gamma^2) I$. Which gives  $\det((\alpha X_1 + \beta X_2 +\gamma X_3))^2=(\alpha^2+\beta^2+\gamma^2)^4$, i.e., $\det((\alpha X_1 + \beta X_2 +\gamma X_3)) = (\alpha^2+\beta^2+\gamma^2)^2$.  
Therefore, $\text{span}\{X_1,X_2,X_3\}$ is a rank-$4$ subspace as desired. Thus, from Theorem~\ref{structure_constant_Camina_algebra} we can construct a 2-step nilpotent real Camina Lie algebra of breadth type $(0,3)$. This Lie algebra is minimally generated by $4$ elements. This Lie algebra can be described using Theorem~\ref{structure_constant_Camina_algebra} as follows:
$$L=\langle  x_1,x_2,x_3,x_4 \mid [x_1,\; x_2] + [x_3,\;x_4] = [x_1,\; x_3] - [x_2,\;x_4] =  [x_1,\; x_4] + [x_2,\;x_3] = 0 \rangle.$$
\end{example}

Now we come to the case of finite fields $k = \mathbb F_q$. For a finite field, $k_{sks}(n)$ can 
be 
determined using the Chevalley-Warning theorem. When $n$ is odd $k_{sks}(n)=0$. When $n$ is even it 
follows from~\cite{ka} Theorem 3.8, that $k_{sks}(n) \leq n/2$. As a consequence, we 
obtain the following for Camina Lie algebras over finite fields.
\begin{theorem}\label{finite_field_Camina_Lie_algebra}
Let $L$ be a $2$-step nilpotent Camina Lie algebra over $\mathbb F_q$ minimally generated by $n$ elements, and of breadth type $(0,m)$. Then, $n\geq 2m$. 
\end{theorem} 
\begin{proof}
Suppose $L$ is a Camina Lie algebra of nilpotency class $2$, minimally generated by $n$ elements and of breadth type $(0,m)$. By Theorem~\ref{structure_constant_Camina_algebra}, there exists a rank-$n$ subspace of $SKS_n(\mathbb F_q)$ of dimension $m$. Now, from the previous theorem we must have $n\geq 2m$.
\end{proof}

\section{Classification of $(0, 2)$ type}\label{02type}

In this section, we discuss the classification of breadth type $(0,2)$ nilpotent Lie algebras. This 
mostly follows from the known results about nilpotent Lie algebras of breadth $2$ which have 
been classified in~\cite{kms} and~\cite{re}. Note that a breadth $2$ Lie algebra 
could be of type $(0,2)$ or $(0,1,2)$. We have the following,
\begin{theorem}\label{breadth_type_(0,2)_classification}
Let $L$ be a finite dimensional nilpotent stem Lie algebra of breadth type $(0, 2)$. Then, $L$ is one of the following:
\begin{enumerate}
\item $L$ is a $2$-step nilpotent Camina Lie algebra with $\dim(L')=2$.
\item $L$ is the $2$-step free nilpotent Lie algebra $\mathcal{L}_2$ generated by three elements. 
\item $L$ is a five dimensional $3$-step nilpotent Lie algebra with a following presentation 
$L=\mathrm{span}\{x_1,x_2, y ,z_1,z_2\}$ with $[x_1, x_2]=y, [x_1, y]=z_1, [x_2, y]=z_2$.
\end{enumerate}
\end{theorem}

We start with the following result which follows easily from the definitions of nilpotency class and breadth type.
\begin{lemma}\label{lemma2}
Let $L$ be a finite dimensional $c$-step nilpotent Lie algebra of breadth type $(0=m_0, m_1, m_2, 
\cdots, m_r)$. Then, $\dim(Z(L))\geq \dim(\gamma_c(L))\geq m_1$, $\dim(\gamma_2(L))\geq m_r$ 
and $\dim(L/Z(L))> m_r$.
\end{lemma}
\noindent The following result is about the relation between Camina Lie algebra and breadth type. The proof is easy but we include for the sake of completeness. 
\begin{lemma}\label{lemma12}
Let $L$ be a finite dimensional nilpotent stem Lie algebra with $\dim(\gamma_2(L))= m$. Then $L$ is 
$2$-step nilpotent Camina Lie algebra if and only if $L$ is of breadth type $(0, m)$.
\end{lemma}
\begin{proof}
Clearly, if $L$ is $2$-step nilpotent Camina Lie algebra with $\dim(\gamma_2(L))= m$, then $L$ is of 
breadth type $(0, m)$. To prove the converse, suppose $L$ is not $c$-step nilpotent with $c\geq 3$. 
Then, $1\leq \dim(\gamma_3(L)) \leq m-1$. For any $x\in \gamma_2(L) \setminus Z(L)\neq \Phi$, we 
get $0 \subsetneq [x, L] \subseteq \gamma_3(L)$, and consequently breadth of $x$ lies in between $1$ 
and $m-1$. Hence $L$ is $2$-step nilpotent and thus $Z(L) = \gamma_2(L)$ (since $L$ is stem). Now 
it follows from the hypothesis  that for all $x\in L \setminus  \gamma_2(L)$, we have $[x, L] = 
\gamma_2(L)$. Thus, $L$ is a $2$-step nilpotent Camina Lie algebra.
\end{proof}
\begin{proof}[\bf{Proof of the Theorem~\ref{breadth_type_(0,2)_classification}}]
Since $L$ is stem, we have $Z(L) \subseteq L'$. From the characterization of breadth $2$ nilpotent 
Lie algebras in~\cite{kms} Theorem 3.1, we have either $\dim (L')=2$, or $\dim (L')=3$ and 
$\dim(L/Z(L))=3$. This gives us following $3$ possibilities, which we will consider one by one.
\begin{enumerate}
\item[(i)] $\dim(L')=2$.
\item[(ii)] $\dim(L) = 6$, $Z(L) = L'$ is of dimension $3$. 
\item[(iii)] $\dim(L) = 5$, $\dim(L')=3$, and $Z(L) = \gamma_3(L)$ is of dimension $2$.
\end{enumerate}
In case (i), $L$ is a $2$-step nilpotent Camina Lie algebra with $\dim(L')=2$. It follows easily 
that any Lie algebra $L$ satisfying case (ii) is isomorphic to the $2$-step free nilpotent Lie 
algebra $\mathcal{L}_2$ generated by three elements. Similarly, any Lie algebra $L$ satisfying case 
(iii) has a presentation  $L = \mathrm{span}\{x_1,x_2,y,z_1,z_2\}$ with $[x_1, x_2]=y, [x_1, y]=z_1, 
[x_2,y]=z_2$.
\end{proof}

Note that the Lie algebra appearing in the case (iii) is the $3$-step free nilpotent Lie algebra 
generated by $2$ elements. Also we see that there are indeed nilpotent Lie algebras of breadth $2$ 
and breadth type $(0,1,2)$. Using~\cite{kms} Theorem 4.1 and 4.5, we can obtain such examples. We 
can further say that a stem nilpotent Lie algebra $L$ is of breadth type $(0, 1, 2)$ if and only if 
$\dim (L') = 2$, and $L$ is not a Camina. We point out that 
Theorem~\ref{breadth_type_(0,2)_classification} works over any field. In particular over fields of 
characteristic $2$, we see that there is a Lie algebra of nilpotency class $3$. Although, there is 
no $2$-group of conjugate type $(1, 2^n)$ for $n\geq 1$.

\section{Nilpotent Lie algebras of breadth type $(0,3)$}\label{proof_of_the_main_theorem}

In this section, we prove our main theorem which classifies breadth type $(0,3)$ nilpotent Lie 
algebras over a finite field of odd characteristics. With a slight abuse of notation, we say $L$ is 
$n$-generated Lie algebra when we mean $L$ is minimally generated by $n$ elements. We explore some properties of $2$-step nilpotent Lie algebras.
\begin{proposition}\label{quotient_breadth_type_(0,m)}
Let $L$ be a $2$-step $(m+1)$-generated nilpotent Lie algebra. Then, $L$ is of breadth type $(0, 
m)$ if and only if $L\cong \mathcal L_m/I$ where $I \subsetneq Z(\mathcal L_m)$, and $I$ does not 
contain any non-zero Lie bracket.
\end{proposition}
\begin{proof}
 Proof is an easy exercise.
\end{proof}
\noindent We now proceed to determine the $(m+1)$-generated quotients $\mathcal L_m/ I$  of breadth type $(0, m)$, where $I$ is an ideal of dimension $1$ or $2$. 
\subsection{When $I$ is of dimension $1$}
We begin with, 
\begin{lemma}\label{main_lemma_1}
Let $k$ be a field. Consider the $2$-step free nilpotent Lie algebra $\mathcal{L}_m$ over $k$ 
generated by the elements $x_1, x_2, \ldots, x_{m+1}$. Let $I$ be a central ideal of dimension $1$ 
such that $\mathcal L_m/I$ is of breadth type $(0, m)$. Then, $\mathcal L_m/I$ is isomorphic to 
$\mathcal L_m/J$ where $J$ is the central ideal of $\mathcal L_m$ spanned by the element $[x_1, x_2] + 
[x_3, x_4] + \cdots + 
[x_{2r-1}, x_{2r}]$ for some $r$ where $2\leq r \leq \ceil{\frac{m}{2}}$.
\end{lemma}
\begin{proof}
It is easy to check $\mathcal L_m/J$ is of breadth type $(0, m)$ where $J$ is as given. 

Conversely, any given central ideal $I$ of dimension $1$, up to an automorphism of $\mathcal{L}_m$, is spanned by an element $x \in Z(\mathcal{L}_m)$ of the form 
$$x = [x_1,x_2] + \alpha_{1,3}^1[x_1,x_3] + \dots + \alpha_{i,j}^1[x_{i},x_{j}]  + \dots + 
\alpha_{m,m+1}^1 [x_m,x_{m+1}], 
$$
where $1\leq i < j \leq m+1$ and $\alpha^1_{i,j}$ are scalars. Now applying the automorphism 
$\phi_1$ (because of Proposition~\ref{automorphism_of_free_nilpotent_Lie_algebra} and 
\ref{set_of_generators}) induced by $x_2 \mapsto x_2 - ( \alpha_{1,3}^1 x_3 + \cdots  +  
\alpha_{1,m+1}^1 x_{m+1}), x_i \mapsto x_i$, for $i \neq 2$ on $x$, it gets mapped to 
$$ \phi_1(x)=[x_1, x_2] + \alpha_{2,3}^2 [x_2, x_3] + \cdots + \alpha_{i,j}^2[x_{i},x_{j}]  + \cdots + \alpha_{m, m+1}^2 [x_{m}, x_{m+1}] $$ 
with $2\leq i < j \leq m+1$. We now apply another automorphism $\phi_2$ induced by the map
$x_1 \mapsto x_1 - (\alpha_{2,3}^2 x_3 + \cdots  +  \alpha_{2,m+1}^2 x_{m+1}),  \ x_i \mapsto x_i$ for $i \ne 1$, we get $x$ of the form 
$$\phi_2 (\phi_1(x)) = [x_1,x_2] + \alpha_{3,4}^3 [x_3, x_4] + \cdots   + \alpha_{i,j}^3 [x_{i}, x_{j}]  + \cdots + \alpha_{m,m+1}^3 [x_m,x_{m+1}],$$
with $3\leq i < j \leq m+1$. From previous proposition, $\mathcal{L}_m/I$ is of breadth type $(0,m)$ if and only if $I$ does not contain any non-zero Lie bracket. In our case, $I$ does not contain any non-zero Lie bracket if and only if at least one of $\alpha_{i,j}^3$, $3\leq i < j \leq m$ is a non-zero scalar. Up to an automorphism of $\mathcal{L}_m$, we can assume that $\alpha_{3,4}^3$ is non-zero. Now consider the automorphism $\phi_3$ induced by the map $x_3 \mapsto (\alpha_{3,4}^3)^{-1} x_3, x_i \mapsto x_i$ for $ i \neq 3$, we get
$$\phi_3 (\phi_2 (\phi _1 (x))) = [x_1, x_2] + [x_3, x_4] + \alpha_{3,5}^4 [x_3, x_5] 
+ \cdots  +\alpha_{i,j}^4[x_{i}, x_{j}]  + \cdots + \alpha_{m, m+1}^4 [x_m, x_{m+1}],$$
with $3\leq i < j \leq m+1$. Finally, if all $\alpha_{i, j}^4$ are $0$, we are done. If not, then finite repetitions of the above process reduces $x$  to the desired form, thus completing the proof.
\end{proof}

\begin{corollary}\label{main_corr_1}
Let $L$ be a $(m+1)$-generated $2$-step nilpotent Lie algebra generated by the elements $\{x_1, \ldots, x_{m+1}\}$ of dimension $\frac{(m+1)(m+2)}{2}-1$. Then, $L$ is of breadth type $(0, m)$ if and only if $L$ is isomorphic to $\mathcal{L}_m/I$, where $I$ is the one dimensional central ideal spanned by  $[x_1, x_2] + [x_3, x_4] + \cdots + [x_{2r-1}, x_{2r}]$ for some $r$ where $2\leq r \leq \ceil{\frac{m}{2}}$.
\end{corollary}
\begin{proof}
Since $L$ is $(m+1)$-generated $2$-step nilpotent Lie algebra of dimension $\frac{(m+1)(m+2)}{2}-1$, $L$ must be isomorphic to a quotient of $\mathcal{L}_m$ by a one-dimensional ideal of $L'$. The proof now follows from the preceding lemma.
\end{proof}

\subsection{When $I$ is of dimension $2$}

We now proceed to characterize the two dimensional central ideals $J$ of $\mathcal{L}_3$ such that the quotient Lie algebra $\mathcal{L}_3/J$ is of breadth type $(0, 3)$.

\begin{lemma}\label{main_lemma_2}
Let $k$ be a field of characteristic not equal to $2$. Let $L$ be a $4$-generated $2$-step nilpotent 
Lie algebra over $k$ of dimension $8$. Consider the Lie algebra $\mathcal{L}_3$ over $k$ with 
generators $x_1, x_2, x_3, x_4$. Then, $L$ is of breadth type $(0, 3)$ if and only $L$ is 
isomorphic to $\mathcal{L}_3/J$, where $J$ is the two dimensional central ideal spanned by  the 
elements $[x_1,x_2] + [x_3,x_4]$ and $[x_1, x_3] + r[x_2, x_4]$, where $r$ is a non-square in 
$k^{\times}$.
\end{lemma}
\begin{proof}
Any given central ideal $J$ of dimension $2$, up to some automorphism $\phi_1$, can be written 
as one of the following two types:
$$J_1 = \mathrm{span}\left\{ [x_1,x_2] + {i_1}[x_1,x_3] + {i_2}[x_1,x_4] + {i_3}[x_2,x_3] + 
{i_4}[x_2,x_4], \; [x_3, x_4]\right\} \text{ or,} $$
\begin{eqnarray*}
J_2 &=& \mathrm{span}\{[x_1,x_2] + i_1^1 [x_1,x_4] + i_2^1 [x_2,x_3] + i_3^1 [x_2,x_4] + i_4^1 
[x_3,x_4], \ \ [x_1,x_3] + j_1^1 [x_1,x_4]\\&& + j_2^1 [x_2,x_3] + j_3^1[x_2,x_4] + j_4^1 
[x_3,x_4]\}
\end{eqnarray*}
\noindent Notice that the breadth of element $\bar{x_3}$ is less than $3$ in $\mathcal{L}_3/J_1$ 
since $[x_3, x_4] \in J_1$. Thus, we can assume that $\phi_1(J) = J_2$. 

Now applying the automorphism $\phi_2$ induced by, $x_2 \mapsto x_2 - i_1^1 x_4, x_3 \mapsto x_3 
- j_1^1 x_4$,  $x_i \mapsto x_i,$ for $i \ne 2, 3$, we have,  
\begin{eqnarray*}
\phi_2(J_2) &=& \mathrm{span} \{ [x_1,x_2] + 
i_2^2 [x_2,x_3] + i_3^2 [x_2,x_4] + i_4^2[x_3,x_4],\ \  [x_1,x_3] + j_2^2 [x_2,x_3] \\ &&+ j_3^2 
[x_2,x_4]  + j_4^2 [x_3,x_4] \}.
\end{eqnarray*}
Now consider the automorphism $\phi_3$ induced by the map, $x_1 \mapsto x_1 +  i_2^2 x_3 +  i_3^2 
x_4$, $x_i \mapsto x_i,$ for $i \ne 1$. Notice that,
$$\phi_3(\phi_2(J_2)) = \textrm{span}\{ [x_1,x_2] + i_4^3 [x_3,x_4],\ \  [x_1,x_3] + j_2^3 
[x_2,x_3] + j_3^3 [x_2,x_4] + j_4^3 [x_3,x_4]\}.$$
By Proposition~\ref{quotient_breadth_type_(0,m)}, $ i_4^3$ has to be a non-zero scalar, otherwise 
$\mathcal{L}_3/\phi_3(\phi_2(J_2))$ can not be of breadth type $(0, 3)$ and consequently 
$\mathcal{L}_3/J$ can not be of breadth type $(0,3)$.

Now, consider the automorphism $\phi_4$ induced by the map $ x_4 \mapsto (i_4^3)^{-1} x_4, x_i 
\mapsto x_i$, for $i \ne 4$. Applying this we get, 
$$\phi_4(\phi_3(\phi_2(J_2))) = \textrm{span} \{ [x_1,x_2] + [x_3,x_4],\ \  [x_1,x_3] + j_2^4 
[x_2,x_3] + j_3^4 [x_2,x_4] + j_4^4 [x_3,x_4] \}.$$
Again, it follows that $j_3^4$ has to be a non-zero scalar, otherwise 
$\mathcal{L}_3/ \phi_4(\phi_3(\phi_2(J_2)))$ can not be of breadth type $(0,3)$, and consequently 
$\mathcal{L}_3/J$ can not be of breadth type $(0,3)$.
Now consider the automorphism $\phi_5$ induced by the map, $x_4 \mapsto x_4 - j_2^4 (j_3^4)^{-1} 
x_3, x_i \mapsto x_i$, for $i \ne 4$. We get,
$$J_3 = \phi_5(\phi_4(\phi_3(\phi_2(J_2)))) = \textrm{span}\{[x_1,x_2] + [x_3,x_4],\ \  [x_1,x_3] 
+ \alpha [x_2,x_4] + \beta [x_3,x_4] \},$$
where $\alpha$ is a non-zero scalar.

Recall from Proposition~\ref{quotient_breadth_type_(0,m)} that $\mathcal{L}_3/J$ is of breadth type 
$(0, 3)$ if and only if $\mathcal{L}_3/J_3$ is of breadth type  $(0, 3)$ if and only if $J_3$ does 
not contain any non-zero Lie bracket. Suppose that $J_3$ contains a non-zero Lie bracket, say 
$[a, b]$. Without loss of generality, we can assume that $a = x_1 + l_1 x_3 + l_2 x_4$ and $b = t_1 
x_2 + t_2 x_3 + t_3 x_4$, with $t_1$ non-zero. Notice that both $l_1$ and $l_3$ have to be zero. 
Moreover if $[a, b] \in J_3 \setminus \{ 0 \}$ then so is  $t_1^{-1}[a, b]$. So, without loss 
of generality we can assume that $a = x_1 - lx_4$ and $b = x_2 + t x_3$ and consequently, $[a, b] 
= [x_1,x_2] + t [x_1,x_3] + l [x_2,x_4] + lt [x_3,x_4]$. Notice that if $[a, b] \in J_3 \setminus \{ 
0 \}$, then 
$[a,b] = ([x_1,x_2]+[x_3,x_4]) + t([x_1,x_3] + \alpha [x_2,x_4]+\beta [x_3,x_4])$ gives 
$[x_1,x_2] + t[x_1,x_3] + l[x_2,x_4] + lt[x_3, x_4] = [x_1,x_2]+ t [x_1,x_3] + t\alpha [x_2,x_4] + 
(1+t  \beta) [x_3,x_4]$. 
Thus, by comparing coefficients we get $ l = t\alpha$ and $lt = 1+ t\beta$ in the field $k$.
Solving this we get
$$t^2\alpha - t \beta - 1 = 0.$$

We can solve this for $t$ if and only if $\beta^2 + 4 \alpha$ is a square (since $char(k) \neq 2$). 
From this we conclude that $\mathcal{L}_3/J$ is of breadth type $(0, 3)$ if and only if $\beta^2 + 
4 \alpha$ is a non-square in the base field. Suppose $\beta^2 + 4 \alpha$ is a non-square in $k$. 
If $\beta = 0$, then $4\alpha$ is a non-square if and only if $\alpha$ is a non-square, and hence 
we are done. So, we assume that $\beta \neq 0$. Let us now fix a non-square $r\in k$. Being 
both $r$ and $\beta^2 + 4 \alpha$ non-square, we get $(4r)^{-1}(\beta^2 + 4 
\alpha)$ is a non-zero square. Let $(4r)^{-1}(\beta^2 + 4 \alpha) = l^2$ and $t = \beta / 2$. Now 
consider the automorphism $\phi_6$ induced by the map $x_1 \mapsto l x_1 + t x_4, x_3 \mapsto l 
x_3 + t x_2, x_i \mapsto x_i$, for $i \ne 1, 3$, it follows that,
\begin{eqnarray*}
&&\phi_6 (J_3)\\
&=& \phi_6 (\langle [x_1, x_2] + [x_3, x_4],  [x_1, x_3] + \alpha [x_2, x_4] + \beta [x_3, x_4] 
\rangle)\\
& =& \langle [l x_1 + t x_4, x_2] + [l x_3 + t x_2, x_4], [l x_1 + t x_4, l x_3 + t x_2] 
+ \alpha [x_2, x_4] + \beta [l x_3 + t x_2, x_4] \rangle \\
& =& \langle l([x_1, x_2] + [x_3, x_4]), [l x_1 + t x_4, l x_3 + t x_2] + \alpha [x_2, x_4] + \beta 
[l x_3 + t x_2, x_4] \rangle\\
& = &\langle l([x_1, x_2] + [x_3, x_4]), lt[x_1, x_2] + l^2[x_1, x_3] +  (\alpha + t\beta - t^2) 
[x_2, x_4] + l (\beta - t) [x_3, x_4] \rangle \\
& =& \langle l([x_1, x_2] + [x_3, x_4]), lt[x_1, x_2] + l^2[x_1, x_3] +  (\alpha + \beta^2/4) 
[x_2, x_4] + lt [x_3, x_4] \rangle\\
& = &\langle l([x_1, x_2] + [x_3, x_4]),  l^2[x_1, x_3] +  rl^2 [x_2, x_4] \rangle\\
& = & \langle l([x_1, x_2] + [x_3, x_4]), [x_1, x_3] +  r [x_2, x_4] \rangle.
\end{eqnarray*}
Now, a Lie algebra $L$ minimally generated by $4$ elements and of dimension $8$ is isomorphic to a 
quotient of $\mathcal L_3$ by a central ideal of dimension $2$. Hence, the proof follows. 
\end{proof}
\begin{corollary}\label{main_corr_2}
Let $k$ be a field with $char(k)\neq 2$. Suppose every element of $k$ has a square root in $k$. 
Then, there does not exist a $4$-generated $2$-step nilpotent Lie algebra of dimension $8$ and 
breadth type $(0, 3)$.
\end{corollary}

\subsection{Proof of the main theorem}

In this subsection we work over the finite field $\mathbb F_q$ of odd characteristics. We begin with 
the characterization of breadth $3$ finite dimensional nilpotent Lie algebras over finite fields of 
odd characteristics given in~\cite{swk} (see Theorem 3.1).
\begin{theorem}\label{breadth_3_characteization}
Let $L$ be a finite dimensional nilpotent Lie algebra over $\mathbb F_q$. Then, $b(L)=3$ if and only 
if one of the following holds,
\begin{enumerate}
\item[(a)] $\dim(L') = 3$ and $\dim(L/Z(L)) \geq 4$.
\item[(b)] $\dim(L') \geq 4$ and $\dim(L/Z(L)) = 4$.
\item[(c)] $\dim(L') = 4$ and there exists a $1$ dimensional central ideal $I$ of $L$ such that $\dim\left ( \dfrac{L/I}{Z(L/I)} \right) = 3$.
\end{enumerate}
\end{theorem} 
\noindent We begin with ruling out the possibility $(c)$ in the above theorem.
\begin{lemma}\label{main_lemma_3}
Let $L$ be a finite dimensional nilpotent Lie algebra over $\mathbb F_q$ of breadth type $(0, 3)$. 
Then, one of following holds: 
\begin{enumerate}[label=(\alph*)]
\item $\dim(L') = 3$ and $\dim(L/Z(L)) \geq 4$.
\item $\dim(L/Z(L)) = 4$ and $\dim(L') \geq 4$.
\end{enumerate}
\end{lemma}
\begin{proof} 
Essentially, we need to show that the case $(c)$ of the Theorem~\ref{breadth_3_characteization} 
does not hold when breadth type is $(0, 3)$. Suppose that there exists a central ideal $I$ of $L$ 
such that $\dim\left( (L/I)/Z(L/I)\right) = 3$. Then, since $L$ is of breadth type $(0, 3)$, it 
follows that $Z(L/I) = Z(L)/I$. Thus, we have 
$$3 = \dim \left( (L/I) / Z(L/I) \right) = \dim \left( (L/I) / Z(L)/I \right) = \dim (L/Z(L)).$$ 
However, from Lemma~\ref{lemma2}, the co-dimension of $Z(L)$ must be greater than $3$, which gives a 
contradiction.
\end{proof}
\begin{lemma}\label{lemma1}
Let $L$ be a Lie algebra and $x,y,z \in L$ such that $[x,z], [y,z] \in Z(L)$. Then, $[[x, y], z] = 
0$.
\end{lemma}
\noindent Now, we are ready to give a proof of 
Theorem~\ref{breadth_type_(0,3)_classification_odd_char_finite_fields}.

\begin{proof}[\bf{Proof of Theorem~\ref{breadth_type_(0,3)_classification_odd_char_finite_fields}}]
We have $Z(L) \subseteq L'$ since $L$ is a stem Lie algebra. We know that $L$ can be of nilpotency 
class $2$ or $3$ (see~\cite{bi} Theorem B). We first show that the case when $L$ is of nilpotency class $3$ can not 
occur. On contrary, let us assume that nilpotency class of $L$ is $3$. Combined with the fact that 
$L$ is a stem Lie algebra, we have $Z(L) \subsetneq L'$. For $x\in L' \setminus Z(L)$, we have 
$[x, L]\subseteq Z(L)$. Thus, $Z(L)$ is of dimension at least $3$, and consequently $L'$ is of 
dimension at least $4$. Therefore, by Lemma~\ref{main_lemma_3}, we have $\dim(L/Z(L)) = 4$. Since 
$L$ is of nilpotency class $3$, $L'$ is Abelian. So, $\dim(L/L')$ has to be at least $3$. By 
earlier observation, we have $\dim(L/L') = 3$ and $\dim(L'/Z(L)) = 1$. Thus, $L$ is generated by $3$ 
elements, say $x, y, z$. Without loss of generality, we can assume that $L' = 
\text{span}\{Z(L), [x, y]\}$. Thus, there exists $z_1, z_2 \in Z(L)$ and scalars $t_1, t_2$ such 
that $[x,z] = t_1[x,y]+z_1, [y,z] = t_2[x,y]+z_2$. Now, by changing our generators suitably (in particular, replacing $z$ by $z+ t_2 x - 
t_1 y$), we have 
$$ L = \langle x,\; y,\; z \rangle,\;\; L' = \text{span}\left\{ Z(L), [x,y]\right\},\;\; 
\text{and}\;\; [x,z],\; [y,z] \in Z(L).$$
Now by Lemma \ref{lemma1}, we have $[[x,y],z] = 0$ and consequently $\langle L',z 
\rangle \subseteq C_L ([x, y])$. Since co-dimension of $\langle L',z \rangle$ in $L$ is $2$, it 
contradicts the hypothesis that $L$ is of breadth type $(0, 3)$. Thus, $L$ can not be of nilpotency 
class $3$.

Now we assume that $L$ is of nilpotency class $2$. Since $L$ is a stem Lie algebra, we have $Z(L) = 
L'$. By Lemma~\ref{main_lemma_3}, we have one of following two possibilities (i) $\dim (L') = 3$ and 
(ii) $\dim (L/L') = 4$. In the case (i), $L$ is a Camina Lie algebra (see Lemma~\ref{lemma12}). In the case (ii), $L$ is 
minimally generated by $4$ elements. We also have $4 \leq \dim(L') \leq 6$. Thus, using 
Corollary~\ref{main_corr_1} and Lemma~\ref{main_lemma_2}, the proof is complete.
\end{proof}

\section{$2$-step $4$-generated nilpotent Lie algebras over other fields}

Our analysis in the previous section can give results about $2$-step $4$-generated nilpotent Lie algebra of width type $(0, 3)$ over a finite fields of even characteristics, $\mathbb C$ and $\mathbb R$. We record the same in this section.

\subsection{Nilpotent Lie algebras over $\mathbb F_q$ of even characteristics}
Let $k=\mathbb{F}_{2^n}$ be the finite field of degree $n$ over $\mathbb{F}_2$. The trace of $x$ is sum of all Galois conjugates of $x$ in $k$ over $\mathbb{F}_2$ given by $Tr(x)=x + x^2 + x^4 + \cdots + x^{2^{n-1}}$. Since, $Tr(x)\in \mathbb{F}_2$ hence $Tr(x)=0$ or $1$. We require the following result. 
\begin{proposition}\label{irreducibilty_criterion_quadratic_polynomial_char2_finite_fields}
Let $f(t) = at^2+ bt + c$ be a polynomial over $k=\mathbb{F}_{2^n}$. Then, $f(t)$ is irreducible if and only if $Tr\left(\frac{ac}{b^2}\right) = 1$. Further, given any $z\in k$ with $Tr(z) = 1$, there exists $m, r \neq 0 \in k$ and $s\in k$ such that $mf(r t + s) = t^2 + t + z$.
\end{proposition}
\noindent Proof of this is an easy exercise. Now, we are ready to look at $4$-generated Lie algebras.
\begin{lemma}\label{main_lemma_2_for_char2}
Let $k=\mathbb{F}_{2^n}$ and $L$ be a $4$-generated $2$-step nilpotent Lie algebra over $k$ of dimension $8$. Consider the Lie algebra $\mathcal{L}_3$ over $k$ with generators $x_1, x_2, x_3, x_4$. Then, $L$ is of breadth type $(0, 3)$ if and only if $L$ is isomorphic to $\mathcal{L}_3/J$, where $J = \textrm{span}\{[x_1,x_2] + [x_3,x_4],\; r[x_1,x_3] + [x_2,x_4] + [x_3,x_4]\}$ with $r\in k$ and $Tr(r)=1$. If $n$ is odd, $r$ can be chosen to be $1$.
\end{lemma}
\begin{proof}
Suppose $J$ is a two dimensional central ideal of $\mathcal L_3$ such that $\mathcal L_3/J$ is of breadth type $(0,3)$. Then, following the proof of Lemma~\ref{main_lemma_2}, we can show that $J$ is of the form 
$$J=\langle [x_1,x_2]+[x_3,x_4], \; [x_1,x_3] + \alpha[x_2,x_4] + \beta [x_3,x_4]\rangle$$
where $x_1, x_2, x_3, x_4$ minimally generate $\mathcal L_3$ and $\alpha, \beta$ are scalars. Once again, proceeding along the same line as in the proof of Lemma~\ref{main_lemma_2}, we can prove that $\alpha,\beta$ can be chosen in such a way that the quadratic equation $\alpha t^2+\beta t + 1$ is irreducible over $k$. Now, using Proposition~\ref{irreducibilty_criterion_quadratic_polynomial_char2_finite_fields}, for any given $z\in k$ with $Tr(z) = 1$, there exists $m, r \neq 0$ in $k$ and $s \in \mathbb{F}$ such that $m f(rt + s) = t^2 + t + z$. This gives $m \beta r = m \alpha r^2 = 1$ and $m (\alpha s^2+\beta s+1)=z$. Now, once again using the technique as in the proof of Lemma~\ref{main_lemma_2} we reduce $J$ to the desired form. For this, let $\phi_1$ be the automorphism of $\mathcal L_3$ induced by $x_2\mapsto r x_2 + s x_3, x_4 \mapsto r x_4 + s x_1$ and $x_i\mapsto x_i$ for $i=1, 3$. Now, we see that $\phi_1$ takes $J$ to $J_1$ given by,
\begin{eqnarray*}
J_1 &=& \langle r([x_1,x_2] + [x_3,x_4]), \; [x_1,x_3] + \alpha [rx_2+sx_3, rx_4 + sx_1] + \beta [x_3, rx_4 + sx_1]\rangle \\ 
&=& \langle [x_1, x_2] + [x_3, x_4], \; (\alpha s^2+ \beta s +1)[x_1, x_3] + \alpha r^2 [x_2, x_4] + \beta r[x_3, x_4]\rangle\\
&=& \langle [x_1, x_2] + [x_3, x_4], \; p^{-1}z [x_1, x_3] + p^{-1}[x_2, x_4] + p^{-1}[x_3, x_4]\rangle\\
&=& \langle [x_1,x_2] + [x_3,x_4],\; z[x_1,x_3] + [x_2,x_4] + [x_3,x_4] \rangle.
\end{eqnarray*} 
	This proves our result. 
\end{proof}
We note that, in the Lemma~\ref{main_lemma_2}, $r$ could be any non-square in the field $k$. This means that two different choices of non-squares in the field $k$ yields isomorphic quotients. In the present case, Lemma~\ref{main_lemma_2_for_char2}, $r$ could be anything with $Tr(r)=1$ and the Lie algebras $\mathcal{L}_3/J$ will be isomorphic. Now we have the following,
\begin{theorem}\label{four_generated_nilpotent_Lie_algebra_(0,3)_even_characteristic}
Let $L$ be a $2$-step $4$-generated nilpotent Lie algebra over $k=\mathbb{F}_{2^n}$. Then, $L$ is of breadth type $(0, 3)$ if and only if $L$ is isomorphic to one of the following:
\begin{itemize}
\item[(i)] $\mathcal{L}_3$ generated by four elements $x_1, x_2, x_3, x_4$.
\item[(ii)] the quotient $\mathcal{L}_3 / I$, where $I$ is the central ideal of dimension $1$ given by $I = \textrm{span}\{[x_1, x_2] + [x_3,x_4] \}$,
\item[(iii)] the quotient $\mathcal{L}_3 / J$, where $J$ is the central ideal of dimension $2$ given by $J = \textrm{span}\{ [x_1,x_2] + [x_3,x_4],\; r[x_1,x_3] + [x_2,x_4] + [x_3,x_4] \}$, where $r\in \mathbb{F}_{2^n}$ with $Tr(r)=1$. If $n$ is odd, $r$ can be chosen to be $1$.
	\end{itemize}
\end{theorem}
\begin{proof}
Since $L$ is $2$-step nilpotent Lie algebra minimally generated by $4$ elements, $L$ is isomorphic 
to $\mathcal L_3/I$ where $I$ is a central ideal of $\mathcal L_3$. Since, Lemma~\ref{main_lemma_1} works over any field we get the Lie algebra in $(ii)$. When $I$ is a central ideal of dimension $2$, we use Lemma~\ref{main_lemma_2_for_char2}. Thus, 
we get the Lie algebra in $(iii)$. Now we consider the case if the central ideal $I$ is of dimension $\geq 3$. If $I$ is of dimension greater than four $L_3/I$ will not be of breadth type $(0,3)$. Thus,  we are left with the case when $L$ is $L_3/I$ where $I$ is of dimension $3$. In this case one can check that $L_3/I$ is a $4$-generated $2$-step nilpotent Camina Lie algebra over $\mathbb{F}_{2^n}$. Now, by Theorem~\ref{finite_field_Camina_Lie_algebra}, any $2$-step nilpotent Camina Lie algebra must be minimally generated by $6$ elements. Thus, we cannot take $I$ to be of dimension $3$. This completes the proof.
\end{proof}
\noindent Note that we have only classified $4$-generated Lie algebras in this case. The more general Lie algebras in even characteristics are more difficult to deal with.

\subsection{Nilpotent Lie algebras over some other fields}
Now we work over $\mathbb C$ and $\mathbb R$.
\begin{theorem}\label{four_generated_nilpotent_Lie_algebra_(0,3)_complex}
Let $L$ be a $2$-step nilpotent $\mathbb C$-Lie algebra minimally generated by $4$ elements. Then $L$ is of breadth type $(0,3)$ if and only if $L$ is isomorphic to one of the following:
\begin{itemize}
\item[(i)] the Lie algebra $\mathcal{L}_3$ generated by four elements $x_1, x_2, x_3, x_4$.
\item[(ii)] the quotient $\mathcal{L}_3 / I$, where $I$ is the central ideal of dimension $1$ given by $I = \textrm{span}\{ [x_1, x_2] + [x_3, x_4]\}$.
\end{itemize}
\end{theorem}
\begin{proof}
The proof follows from arguments used in the previous theorem using Corollary~\ref{main_corr_1}, Corollary~\ref{main_corr_2}, and Theorem~\ref{complex_Camina_Lie_algebra}. 
\end{proof}
\begin{theorem}\label{four_generated_nilpotent_Lie_algebra_(0,3)_real}
Let $L$ be a $4$-generated $2$-step nilpotent $\mathbb R$-Lie algebra Then, $L$ is of breadth type $(0,3)$ if and only if $L$ is isomorphic to one of the following:
\begin{itemize}
\item[(i)] a $2$-step Camina Lie algebra $L$ with derived subalgebra $L'$ of dimension $3$. 
\item[(ii)] The Lie algebra $\mathcal{L}_3$ generated by four elements $x_1,x_2,x_3,x_4$.
\item[(iii)] The quotient $\mathcal{L}_3 / I$, where $I$ is the central ideal of dimension $1$ given by $I = \textrm{span}\{ [x_1, x_2] + [x_3, x_4]\}$,
\item[(iv)] The quotient $\mathcal{L}_3 / J$, where $J$ is the central ideal of dimension $2$ given by $J = \textrm{span} \{[x_1,x_2] + [x_3,x_4],\; [x_1,x_3] + t[x_2, x_4]\}$, where $t<0$.
	\end{itemize}
\end{theorem}
\begin{proof}
Once again the proof of this simply follows from Corollary~\ref{main_corr_1}, Lemma~\ref{main_lemma_2}, and Example~\ref{real_Camina_four_generated}.
\end{proof}

In the previous Theorem, we see that there could be a $3$ dimensional central ideal $I$ of $\mathcal L_3$ such that $\mathcal L_3/I$ is of breadth type $(0, 3)$. This can be obtained over a larger class of fields as follows.
\begin{proposition}
Let $k$ be a field where $-1$ is not a sum of two squares. Then, there exists a $4$-generated $2$-step nilpotent Camina Lie algebra over $k$ of breadth type $(0,3)$.
\end{proposition}
\begin{proof}
We will show that $\mathcal L_3/I$ where $I$ is the three dimensional central ideal given by $I=\textrm{span}\{ [x_1,x_2]+[x_3,x_4], \; [x_1,x_3]-[x_2,x_4], \; [x_1,x_4] + [x_2,x_3] \} $ is of breadth type $(0, 3)$ over $k$. From Proposition~\ref{quotient_breadth_type_(0,m)}, it is enough to prove that $I$ does not contain a non-zero Lie bracket. Suppose, on the contrary $I$ contains a non-zero Lie bracket, say $[x, y]$ for $x, y\in \mathcal{L}_3$. Without loss of generality, we can take $x = x_1+ \alpha_3 x_3 + \alpha_4x_4$ and $y = x_2+ \beta_3 x_3 + \beta_4x_4$. Then, 
$$[x, y] = [x_1, x_2] + \beta_3[x_1,x_3] + \beta_4[x_1, x_4] - \alpha_3[x_2, x_3] -\alpha_4 [x_2, x_4] + (\alpha_3\beta_4-\beta_3\alpha_4) [x_3, x_4].$$
This gives $\beta_3 = \alpha_4$ and $\beta_4 = -\alpha_3$. This yields $-(\beta_3^2 + \beta_4^2) = 1$. Which contradicts our assumption that $-1$ is not a sum of two squares in $k$. This completes the proof.
\end{proof}
\noindent Thus, in the Theorem~\ref{four_generated_nilpotent_Lie_algebra_(0,3)_real}, we note that the nilpotent real Camina Lie algebras generated by $4$ elements and of breadth type $(0,3)$ are obtained as a quotient of $\mathcal L_3$ by a $3$ dimensional central ideal. Example~\ref{real_Camina_four_generated} shows that there exists a $3$ dimensional central ideal $I$ of $L_3$ such that $L_3/I$ is of breadth type $(0, 3)$. This ideal $I$ is given as follows:
	$$I=\textrm{span}\{ [x_1,x_2]+[x_3,x_4], \; [x_1,x_3]-[x_2,x_4], \; [x_1,x_4] + [x_2,x_3] \} $$
where $L_3$ is generated by $x_1, x_2, x_3, x_4$.

\end{document}